\documentclass[a4paper]{amsart}
\usepackage{amssymb}
\usepackage{verbatim}
\usepackage{amscd}

\numberwithin{equation}{section}
\theoremstyle{plain}
\newtheorem{theorem}[equation]{Theorem}

\newtheorem{lemma}[equation]{Lemma}
\newtheorem*{(DQ1)}{(DQ1)}
\theoremstyle{definition}

\theoremstyle{remark}

\begin{document}
\title {On flow-equivalence of R-graph shifts}
\author{Wolfgang Krieger}
\begin{abstract}
We show that Property $(A)$ of subshifts and the semigroup, that is associated to subshifts with Property (A), are invariants of  flow equivalence. 
We show for certain  $\mathcal R$-graphs 
that their isomorphism is implied by the  flow equivalence of their $\mathcal R$-graph shifts. 
\end{abstract}
\maketitle

\section{Introduction}
Let $\Sigma$ be a finite alphabet, and let $S_\Sigma$ be the shift 
on the shift space $\Sigma^{\Bbb Z}$,
$$
S_\Sigma((x_{i})_{i \in \Bbb Z}) =  (x_{i+1})_{i \in \Bbb Z}, 
\qquad 
(x_{i})_{i
\in \Bbb Z}  \in \Sigma^{\Bbb Z}.
$$
$S_\Sigma$-invariant closed subsets $X$ of $\Sigma^{\Bbb Z}$  (more precisely, 
with $S_X$ denoting the restriction 
of $S_\Sigma$ to $X$,
the dynamical systems $( X , S_X )$  are called  subshifts. These are the subject of  symbolic dynamics. For an introduction to symbolic dynamics see \cite {Ki} or  \cite{LM}. 

A word is called admissible for a subshift $X\subset\Sigma^{\Bbb Z}$  if it appears in a point of $X$. We denote the set of admissible words of a subshift $X\subset\Sigma^{\Bbb Z}$ by 
$\mathcal L(X)$. The language $\mathcal L(X)$ is  factorial and bi-extensible, and every factorial and bi-extensible language is the set of admissible words af a unique subshift.

Let $\bullet$ be a symbol that is not in  $\Sigma$, and
consider  a subshift $X \subset\Sigma^{\Bbb Z}.$ Denote by $\varphi^{(\sigma)}$ the mapping that assigns to a word $a\in \mathcal L (X)$ the word that is obtained from $a$ by carrying out the substitution that  replaces the symbol $\sigma$ by the word $\sigma\bullet$. The set of subwords of the words in 
$ \varphi^{(\sigma)}(\mathcal L (X))$ is a factorial and bi-extensible language, and we denote the subshift that it determines by $X^{(\sigma)}$. One says that the subshift $X^{(\sigma)}$ arises from the subshift $X$ by symbol expansion. In Section 2  we describe some effects of symbol expansion.

Subshifts $X \subset  \Sigma^{\Bbb Z}$ and $\widetilde{X} \subset 
\widetilde{ \Sigma}^{\Bbb Z}$ are called flow equivalent if there exists a sequence $Z_k, 1 \leq k \leq K, K \in \Bbb N,$ of subshifts, such that $ X = Z_1 $ and $\widetilde{X} = Z_K $,  and such that $Z_k $ is topologically conjugate to $Z_{k-1}$, or $Z_k $ is obtained from $Z_{k-1}$ by symbol expansion, or  $Z_{k-1} $ is obtained from $Z_k $ by symbol expansion, $1<k \leq K$. Flow equivalence was introduced by Parry and Sullivan in 1975 \cite {PS}. Next to topological conjugacy it is one of the fundamental equivalence relations for subshifts.

The notions of  $\mathcal R$-graph, $\mathcal R$-graph semigroup, and  
$\mathcal R$-graph shift were introduced in \cite {Kr2}. The class of $\mathcal R$-graph shifts  contains the class of Markov-Dyck shifts
 \cite {M3}.
In Section 5 we show for certain $\mathcal R$-graphs, that the flow equivalence of their 
$\mathcal R$-graph shifts implies their isomorphism. This extends a result of Costa and Steinberg \cite{CS} for Markov-Dyck shifts. The proof  uses  Property $(A)$  and  the semigroup that is associated to subshifts with Property  $(A)$ \cite{Kr1}.
In Section 3 we prove invariance under flow equivalence of Property $(A)$  and in Section 4  we prove invariance under flow equivalence of the associated semigroup.
For 
 an extension of the theory beyond subshifts with Property $(A)$
 see Costa and Steinberg \cite {CS}.

In Section 5 we consider $\mathcal R$-graph shifts. In \cite{HK} there was given a mecessary and sufficient condition
for an $\mathcal R$-graph to have an $\mathcal R$-graph shift with Property $(A)$, whose associated semigroup is the  $\mathcal R$-graph semigroup of the $\mathcal R$-graph. Under this condition we prove in Section 5, that the flow equivalence of  $\mathcal R$-graph shifts implies the isomorphism of  the underlying $\mathcal R$-graphs. 

\section{Symbol expansion}
 
We introduce  notation for subshifts $X \subset \Sigma^{\Bbb Z}$. The $S_X$-orbit of a point $x\in X$ we denote by $O_X(x)$, and
for an $S_X$-invariant set $A \subset X$ we denote the set of $S_X$-orbits in $A$ by $\Omega(A)$. The period of a periodic point $p\in X$ we denote by $\pi(p)$.
We set
$$
	x_{[i,j]}Ê=Ê(x_{k})_{i\leq k \leq j},   
$$
and
$$
X_{[i,j]}Ê=Ê\{ x_{[i,j]} : x\in X \}, \quad i,j \in \Bbb Z, i \leq j, \qquad  x \in X, 
$$
and we use similar notation in the case that indices range in semi-infinite intervals.
(The elements in $X_{[i,j]}, X_{[i,\infty)},X_{(\infty,i)}$ can be identified with the words they carry. From the context it becomes clear, if such an identification is made.)
We set
$$
\Gamma^+_X(a) = \{ x^+ \in X_{(j, \infty)}: a x^+ \in X_{[i, \infty)}\}, \quad 
a \in  x_{[i,j]}, \quad i,j \in \Bbb Z, i \leq j.
$$
The notation $\Gamma^-$ has the symmetric meaning.
We also set
$$
\omega^+_X(a) = \bigcap_{x^-\in \Gamma^-(a)} \{ x^+ \in \Gamma^+(a):x^-a x^+ \in X\}, \quad
 a \in  x_{[i,j]}, \quad i,j \in \Bbb Z, i \leq j.
$$
The notation $\omega^-$ has the symmetric meaning.
And we set
$$
\Gamma_X(a) = \{ (x^-,x^+) \in \Gamma^-(a)   \times\Gamma^+(a)  : x^-a x^+ \in X\}, \quad
a \in  x_{[i,j]}, \quad i,j \in \Bbb Z, i \leq j.
$$

Let $\sigma \in \Sigma$,  let $\bullet$ be a symbol that is not in  $\Sigma$, and
consider  for a  subshifts $X \subset\Sigma^{\Bbb Z}$ the subshift  
$X^{(\sigma)} \subset(\Sigma\cup\{ \bullet \})^{\Bbb Z}$.
We denote by $\varphi^{(\sigma)}_-$($\varphi^{(\sigma)}_+$) the mapping that assigns to
$x^-\in X_{(-\infty, 0)}$($x^+\in X_{[0, \infty)}$) the point n 
$X^{(\sigma)}_{(-\infty, 0)}$($X^{(\sigma)}_{[0, \infty)}$) that is obtained
from $x^-$($x^+$) by carrying out the substitution that replaces the symbol $\sigma$ by the word $\sigma\bullet$. Also we denote by $ \varphi^{(\sigma)}$ the mapping that assigns to a point $x\in X$
 the point  in $X^{(\sigma)}$, that is given by
 $$
\varphi^{(\sigma)}(x)_{(-\infty, 0)} =   \varphi_-^{(\sigma)}(x_{(-\infty, 0)}),
\quad
\varphi^{(\sigma)}(x)_{[0,-\infty)} =   \varphi_+^{(\sigma)}(x_{[0,-\infty)}).
$$
One observes that
$$
\varphi^{(\sigma)}(O_X(x)) \subset O_{X^{(\sigma)}}( \varphi^{(\sigma)} (x)), \quad x \in X.
$$
For precision we note, that one has, with $\ell^-(x,n)$($\ell^+(x,n)$) denoting the length of 
$ \varphi^{(\sigma)} (x_{[-n, 0)}  )  $($ \varphi^{(\sigma)} (x_{[0,n)}  ) $), that
\begin{align*}
&\varphi^{(\sigma)}(S_X^{-n}(x)) = S_{X^{(\sigma)}}^{-\ell^-(x,n)}( \varphi^{(\sigma)} (x)), \\
&\varphi^{(\sigma)}(S_X^{-n}(x)) = S_{X^{(\sigma)}}^{-\ell^-(x,n)}( \varphi^{(\sigma)} (x)),
\quad n \in \Bbb N.
\end{align*}
Also,
$$
\varphi^{(\sigma)}  (X) \cup S_{X^{(\sigma)}}(  \varphi^{(\sigma)}  (X)).
$$
We denote the bijection of $\Omega(X)$ onto $\Omega(X^{(\sigma)})$ that assigns to the 
$S_X$-orbit of $x\in X$ the $S_{X^{(\sigma)}}$\negthinspace-orbit of $ \varphi^{(\sigma)} (x)$ by $\xi_{\sigma}$.

\begin{lemma} 

For a subshift $X \subset\Sigma^{\Bbb Z}$ and for $\sigma\in \Sigma, a \in \mathcal L(X),$ one has
$$
\varphi_+^{(\sigma)}(\omega_{X}^+(a)) = 
 \omega_{X^{(\sigma)}}^+ (\varphi ^{(\sigma)} (a)).
$$

\end{lemma}
\begin{proof} 
We prove that $  \varphi ^{(\sigma)}(\omega_{X}^+(a))  \subset  \omega_{X^{(\sigma)}}^+ (\varphi ^{(\sigma)} (a)) $. Let $x^+\in \omega^+_X(a)$, and let
$$
y^- \in \Gamma_{X^{(\sigma)}}^-(   \varphi ^{(\sigma)} (a)) .
$$
It follows from $ \varphi ^{(\sigma)}(a)_0\neq\bullet$, that $y^-_{-1}   \neq \sigma$, and one sees that
$ y^- $ is in the image of $ \varphi ^{(\sigma)}_-  $. Its inverse image $x^-$ under
$ \varphi ^{(\sigma)}_-  $ is in $\Gamma^-_X(a)$. It follows that $  x^-ax^+ \in X$, and therefore 
$$
\varphi ^{(\sigma)}(x^-ax^+)  = y^-  \varphi ^{(\sigma)} (a)  \varphi_+ ^{(\sigma)} (x^+) \in X^{(\sigma)},
$$
and this means that $  \varphi_+ ^{(\sigma)} (x^+) \in \omega_{X^{(\sigma)}}^+ (\varphi ^{(\sigma)} (a)) $.

For the converse one has a similar argument.
\end{proof}

\begin{lemma} 
For a subshift $X \subset\Sigma^{\Bbb Z}$ and for $\sigma\in \Sigma, b, b^{\prime} \in \mathcal L(X),$ one has
$$
\Gamma_X (b) =   \Gamma_X ( b^{\prime}),
$$
if and only if
$$
\Gamma_{X^{(\sigma)}} (\varphi ^{(\sigma)}(b)) =
 \Gamma_{X^{(\sigma)}} (\varphi ^{(\sigma)}( b^{\prime})).
$$
\end{lemma}
\begin{proof} 
The lemma follows from
\begin{align*}
\Gamma^+_{X^{(\sigma)}}(\varphi^{(\sigma)}(a)  ) \subset\varphi_+^{(\sigma)}
(\Gamma^+_X(a)),
\ \
\Gamma^-_{X^{(\sigma)}}(\varphi^{(\sigma)}(a)  ) \subset \varphi_-^{(\sigma)}
(\Gamma^-_X(a)),\quad
a \in \mathcal L(X). \qed
 \end{align*}
\renewcommand{\qedsymbol}{}
\end{proof}

\section{Property (A)}

Given a subshift $X\subset \Sigma ^\Bbb Z$ we define a subshift of finite type  $A_{n}(X)$ by 
\begin{multline*}
 A_{n}(X) = \\
  \bigcap _{i\in \Bbb Z}
( \{x  \in X: 
x_{[i,\infty)} \in \omega_X^{+} (x_{[i - n, i)})\}\cap \{x  \in X: 
x_{(-\infty,i]} \in \omega_X^{-} (x_{(i,i + n]})\}),  \ \
n \in \Bbb N,
\end{multline*}
and we set
$$
A(X) = \bigcup_{n \in \Bbb N} A_{n}(X).
$$

\begin{lemma}
For a subshift  $X \subset\Sigma^{\Bbb Z}$, and for $\sigma \in \Sigma$, one has
\begin{align*}
\xi_{\sigma}(\Omega(A_n(X))) \subset \Omega(A_{2n}(X^{(\sigma)})), \qquad n \in \Bbb N, \tag 1
\end{align*}
and
\begin{align*}
\xi_{\sigma}^{-1}(\Omega(A_n(X^{(\sigma)}))) \subset \Omega(A_{n}(X)), \qquad n \in \Bbb N.
\tag 2
\end{align*}
\end{lemma}
\begin{proof}
We show (1).
Let $n\in \Bbb N$, let $x\in A_n(X)$, and let $i\in \Bbb Z$. Let $\mu$ be the number of times that the symbol $\bullet$ appears in $\varphi^{(\sigma)}(x)_{[i,i+2n]}$.
Assume that neither $x^{(\sigma)}_i = \bullet$, nor $x^{(\sigma)}_{i+2n-1} = \sigma$. 
Then
$$
\varphi^{(\sigma)}( x_{[i,i+ 2n-\mu)})  =\varphi^{(\sigma)}(x)_{[i,i+2n]}.
$$
From
$$
x_{[i+2n-\mu, \infty )} \in \omega^+_X(x_{[i, i+2n-\mu)}),
$$
it follows then by Lemma 2.1, that
\begin{align*}
 \varphi^{(\sigma)}(x)_{[i+2n, \infty )} \in \omega^+_{X^{(\sigma)}}(\varphi^{(\sigma)}(x)_{[i,i+2n)}).     \tag 3
\end{align*}
In the case that $x^{(\sigma)}_i = \bullet$, necessarily  $x^{(\sigma)}_{i-1} = \sigma$, 
and in the case that  $x^{(\sigma)}_{i+2n-1} = \sigma$, necessarily 
$x^{(\sigma)}_{i+2n}= \bullet$,
and in both cases it is seen that (3) also holds.

For (2) one has a similar argument.
\end{proof}

We recall from \cite {Kr1} the definition of Property (A).
For $n \in \Bbb N$ a subshift $X \subset  \Sigma ^{\Bbb 
Z},$ 
has property $(a, n, H),H \in \Bbb N,$  if for $h,
\widetilde {h} \geq 3H$ and for
$
I_-,  I_+,\widetilde{I}_-,  \widetilde{I}_+ \in \Bbb Z,
$
such that
$$
I_+-I_-, \widetilde{I}_+-\widetilde{I}_- \geq 3H,
$$
and for
$$
a \in A_{n}(X)_{(I_-,  I_+ ]}, \quad \widetilde{a} \in A_{n}(X)_{(\widetilde{I}_-, 
 \widetilde{I}_+ ]}, 
$$
such that
$$
a_{(I_-, I_- +  H]} = \tilde{a}_{(\widetilde{I}_-, \widetilde{I}_ ++ H]},\quad
a_{(  I_+  - H,   I_+ ]} = \tilde{a}_{( \widetilde{I}_+ - H,  \widetilde{I}_+]},
$$
one has that 
$$
\Gamma_X(a ) =\Gamma_X(\tilde{a}). 
$$
It is assumed, that $A(X) \neq \emptyset$.
The subshift $X \subset  \Sigma ^{\Bbb Z}$
has property $(A)$ if there are $H_{n}, n \in  \Bbb N$, such
that $X$ has the properties $(a, n, H_{n}), n \in \Bbb N $.

\begin{theorem} 
For a subshift $X \subset\Sigma^{\Bbb Z}$ and for $\sigma\in \Sigma,$  one 
has that $X$ has Property $(A)$ if and only if $X^{(\sigma)}$ has Property $(A)$.
\end{theorem}
\begin{proof} 
The theorem follows from Lemma 2.2 and Lemma 3.1.	
\end{proof}

\section{The associated semigroup}

Consider a subshift $X \subset  \Sigma ^{\Bbb Z}$ with Property $(A)$.
We denote the set of periodic points in $A(X)$ by $P(A(X))$.   We introduce a preorder relation 
$\gtrsim \negthinspace \negthinspace{(X)}$
into the set $P(A(X))$ where for 
$q, r \in  P(A(X)),q\gtrsim\negthinspace \negthinspace{(X)}\thinspace r ,$
 means that there exists a point in $A(X)$ that is left asymptotic to the orbit of $q$ and right asymptotic to the orbit of $r$. The equivalence relation on $P(A(X))$ that results from the preorder relation 
 $\gtrsim\negthinspace \negthinspace(X)$ we denote by 
 $\approx\negthinspace \negthinspace(X)$.
We denote the set of $\approx\negthinspace \negthinspace(X)$-equivalence classes by $\frak P(X)$.

\begin{lemma}
For a subshift $X \subset\Sigma^{\Bbb Z}$, for $\sigma \in \Sigma$.  $q,r\in P(A(X))$, and for $\sigma\in \Sigma,$ one has
$$
q \gtrsim \negthinspace \negthinspace(X) r,
$$
if and only if
$$
\varphi^{(\sigma)} (q) \gtrsim \negthinspace \negthinspace(X^{(\sigma)})\varphi^{(\sigma)} (r) .
$$
\end{lemma}
\begin{proof}
This follows from Lemma 3.1.
\end{proof}

We  recall  the construction of the associated semigroup. For a property $( A)$ subshift $X \subset \Sigma ^{ \Bbb Z}$ we denote by  $Y(X)$ the set of points in $X$
that are left asymptotic to a  point in $P(A(X))$ and also right-asymptotic to a  point 
in $P(A(X))$ .  Let $y, \tilde{y}\in Y(X),$ let $y$ be left asymptotic to $q \in P(A(X))$ and right 
asymptotic to $r \in P(A(X) ) ,$ and let 
$\tilde {y}$ be left asymptotic to $ \tilde{q} \in P(A(X))$ and right
asymptotic to $ \tilde{r} \in P(A(X))$. Given that 
$X$ has the properties $(a, n, H_{n}), n \in \Bbb N,$ we say
that $y$ and $\tilde {y}$ are equivalent, $y \approx\negthinspace(X)\ \tilde {y}$, if
$q  \approx\negthinspace(X)\ \tilde{q}$ and $ r  \approx\negthinspace(X)\ \tilde{r}$, and  
if for $n \in \Bbb N$ such that $ q, r, \tilde {q}, \tilde {r} 
\in P(A_n(X))$ and  for $I, J, \tilde
{I}, \tilde {J} \in \Bbb Z,  I < J, \tilde {I}< \tilde
{J},$ such that
$$
y_{(- \infty, I]} = q_{(- \infty, 0]},\quad  y_{(  J,\infty)} = r_{(  0,\infty)},
$$
$$
\tilde{y}_{(- \infty, \tilde{I}]} = \tilde{q}_{(- \infty, 0]},\quad \tilde{ y}_{(  \tilde{J},\infty)} = \tilde{r}_{( 
 0,\infty)},
$$
one has for $h\ge3 H_{n}$ and for 
$$ 
a \in X_{(I  -  h,J + h]}, \quad
\tilde {a} \in X_{(\tilde {I} -  h,\tilde {J} + h]}, 
$$
such that  
$$
a _{(I-  H_{n} ,  J + H_{n} ]} = y_{(I - H_{n} ,  J + H_{n} ]} ,\quad
 \tilde {a} _{(\tilde {I} -  H_{n} ,  \tilde {J} + H_{n} ]} =
\tilde{y}_{(\tilde {I} -  H_{n} ,  \tilde {J} + H_{n} ]},
$$
$$
a _{(I - h ,  I  - h + H_{n})} = \tilde {a} _{(\tilde {I} -
h , \tilde {I}  - h+ H_{n} )},
$$
$$
a _{(J + h- H_{n} ,  J  + h ]} =
 \tilde {a} _{(\tilde {J }+ h - H_{n}  , 
\tilde { J}  + h ]},
$$
and such that
$$
a_{(I - h, I]} \in A_{n}(X)_{(I - h, I]} , \quad
\tilde{a}_{( \tilde{J}- h, \tilde{I}]}\in A_{n}(X)_{( \tilde{J}- h, \tilde{I}]}\ ,
$$
$$
 a_{(J , J+h]} \in A_{n}(X)_{(J , J+h]}  , \quad
  \tilde{a}_{( \tilde{J}  , \tilde{J}+h]}  \in A_{n}(X)_{( \tilde{J}  , \tilde{J}+h]} ,
$$
that
$$
\Gamma_X(a) = \Gamma_X(\tilde {a}).
$$

To give $[Y(X)]_{ \approx\negthinspace(X)\ }$ the
structure of a semigroup, let $u, v \in Y(X)$, let $u$ be right asymptotic to
$q \in P(A(X))$ and let  $v$ be left asymptotic to
$r \in P(A(X))$. 
If here 
$q \gtrsim\negthinspace(X) r$, then $[u]_{ \approx\negthinspace(X)\ }[v]_{ \approx\negthinspace(X)\ }$ is set equal to $[y]_{ \approx\negthinspace(X)}$, where $y$ is any point in $Y$ such that there are
$n \in \Bbb N, I, J ,\hat {I},  \hat {J} \in \Bbb Z , I < J, \hat 
{I}<   \hat {J}, $ such that $ q, r \in A_{n}(X),$  and such that
$$
u_{(I, \infty)} = q_{(I, \infty)}, \quad
v_{(-\infty, J]} = r_{(-\infty, J]},
$$  
$$
y_{(- \infty,\hat{I} + H_{n}]}= u_{(- \infty,I + H_{n}]} ,\quad
y_{ (\hat{J}  - H_{n}, \infty)}= v_{ (J - H_{n}, \infty)},
$$
and 
$$
y_{(\hat{I}  , \hat{J}  ]}\in A_{n}(X)_{(\hat{I}  , \hat{J}  ]},
$$
provided that such a point $y$ exists. If such a point $y$ does not exist,  
$[u]_{ \approx\negthinspace(X)\ }[v]_{ \approx\negthinspace(X)\ }$ is  equal to zero.
Also, in the case that
one does not have $q \gtrsim\negthinspace(X) r,[u]_{\approx}[v]_{ \approx\negthinspace(X)\ }$ is equal to zero.

Consider a subshift $X\subset \Sigma^{\Bbb Z}$ with Property $(A)$.\ \negthinspace For 
$\frak p \in \frak P(X)$ we choose a $d^{(\frak p)} \in \frak p$, and we set 
$$
\mathcal D = \{d^{(\frak p)}  : \frak p \in \frak P(X)  \}.
$$
In order to facilitate the proof of its invariance under flow equivalence we give an alternate description of the semigroup  that is associated to $X$ in terms of the system $\mathcal D\subset Y_X$ of representatives of the equivalence relation $\approx\negthinspace(X)$. 
For $y \in O_X(d^{(\frak p)} ), \frak p \in \frak P(X),$ we define a $J(y, d^{(\frak q)})\in \Bbb Z$ by
$$
S_X^{-J(y, d^{(\frak p)}  )}(y) =   d^{(\frak p)}, \quad 
 0 \leq    \pi(d^{(\frak p)}) <\pi(d^{(\frak p)}).
$$
For $ \frak p \in \frak P(X)$ we set
$$
H( d^{(\frak p)}  ) = \min  \thinspace \{H \in \Bbb N: \Gamma_X( \frak p_{[0,H \pi( \frak p ))  }   ) =
  \Gamma_X( \frak p_{[0,(H +1) \pi(d^{(\frak p)} )) }   ) \}.
$$
We denote by $Y_X^-(\mathcal D)$, the set of points in $Y_X$, that are left  asymptotic to the orbit of a point in $\mathcal D$, and also right asymptotic to the orbit of a point in $\mathcal D$.
More precisely, we denote by $Y_X^-(d^{(\frak p)})$($Y_X^+(d^{(\frak p)} )$), the set of points in $Y_X$, that are left (right) asymptotic to the orbit of $d^{(\frak p)}, \frak p \in \frak P(X)$.
For
$$
y \in Y_X^-(d^{(\frak q)}q)\cap Y_X^+(d^{(\frak r)} ),  \qquad  \frak q, \frak r \in \frak P(X),
$$
we set
$$
I^-(y)  = \begin{cases}J(y, d^{(\frak q)}), &\text {if $y \in  O_X( d^{(\frak q)})$,}\\
\max \{I \in \Bbb Z : y_{(-\infty, I)} = d^{(\frak q)}_{(-\infty, 0)}   \}, & \text{if $y \not \in  
O_X( d^{(\frak q)}),$}
\end{cases}
$$
$$
I^+(y)  = \begin{cases}J(y, d^{(\frak r)}), &\text {if $y \in  O_X( d^{(\frak r)})$,}\\
\min \{I \in \Bbb Z : y_{[ I, \infty)} = d^{(\frak r)}_{[ 0, \infty)}   \}, & \text{if $y \not \in  
O_X( d^{(\frak r)}).$}
\end{cases}
$$
We say that $O,O^{\prime} \in \Omega(Y_X^{(D)})$ are $\approx\negthinspace(D)$-equivalent,  
if $O$ and $O^{\prime}$ are left asymptotic to the same periodic obit, and also right asymptotic to  the same periodic obit,
and if, with $\frak q \in \frak P$ such that $y$ and $y^{\prime}$ are right asymptotic to the orbit of $d^{(\frak q)}$ and with $\frak r \in \frak P$ such that $y$ and $y^{\prime}$ are left asymptotic to the orbit of $d^{(\frak r)}$,  there exist $y\in O$ and 
$y^{\prime} \in O^{\prime}$ such that
\begin{multline*}
\Gamma_X ( d^{(\frak q)}_{[0,H(d^{(\frak q)})\pi(d^{(\frak q)} ))}  y_{[I^-(y), I^+(y))} 
d^{(\frak r)}_{[0,H(d^{(\frak r)})\pi(d^{(\frak r)} ))}) = \\
\Gamma_X( d^{(\frak q)}_{[0,H(d^{(\frak q)})\pi(d^{(\frak q)} ))}   
y^{\prime}_{[I^-(y^{\prime}), I^+(y^{\prime}))} 
d^{(\frak r)}_{[0,H(d^{(\frak r)})\pi(d^{(\frak r)} ))}  ).
\end{multline*}
To give $\Omega(Y_X^{(D)})$ the structure of a semigroup, let $\frak q,\frak p,\frak r \in \frak P$, and let for
points
$$
u\in Y_X^-(\frak q)\cap Y_X^+(\frak p ), \qquad  v\in Y_X^-(\frak p)\cap Y_X^+(\frak r ),
$$
in case, that the word
\begin{align*}
d^{(\frak q)}_{[0, H(\frak  d^{(\frak q)}) \pi( d^{(\frak q)}))}
y_{[I^-(u), I^+(u))}
 d^{(\frak p)}_{[0, H( d^{(\frak p)}) \pi(  d^{(\frak p)}))}
y_{[I^-(v), I^+(v))}
 d^{(\frak r)}_{[0, H( d^{(\frak r)}) \pi( d^{(\frak r)}))}, \tag 3
\end{align*}
is admissible for $X$, let
a point $y[u, v] \in Y_X^-(\frak q)\cap Y_X^+(\frak r ) $ be given by
$$
y[u, v]_{(-\infty, 0)} =  d^{(\frak q)}_{(-\infty, 0)},
$$
and
$$
y[u, v]_{[0, \infty)} = d^{(\frak q)}_{[0, H(\frak q) \pi( \frak q))}
y_{[I^-(u), I^+(u))}
 d^{(\frak p)}_{[0, H(d^{(\frak p)}) \pi( d^{(\frak p)})}
y_{[I^-(v), I^+(v))}
 d^{(d^{(\frak r)})}_{[0, \infty)}.
$$
and set
\begin{align*}
[O(u)]_{\approx (D)}[O(v)]_{\approx (D)} = [O(y[u, v])]_{\approx (D)}.
\end{align*}
In case, that the word (3) is not admissible for $X$,  set
\begin{align*}
[O(u)]_{\approx (D)}[O(v)]_{\approx (D)} = 0. 
\end{align*}
Also, for
$\frak q,\frak r \in \frak P$, if
 $$
Y_X^-(\frak q) \cap A(X) \cap Y_X^+(\frak r) \neq \emptyset,
$$
define a $ \approx\negthinspace(D) $-equivalence class $\gamma ( \frak q ,  \frak r  )$ by
$$
\gamma ( \frak q ,  \frak r  ) = [O( y )]_{\approx(D)}, \quad y \in Y_X^-(\frak q) \cap A(X)
 \cap Y_X^+(\frak r).
$$
As a consequence of  Property $(A)$ of $X$ the $ \approx\negthinspace(D) $-equivalence class $ \gamma ( \frak q ,  \frak r  ) $ is well defined. If
$$
Y_X^-(\frak q) \cap A(X) \cap Y_X^+(\frak r) = \emptyset,
$$
 set
$$
\gamma ( \frak q ,  \frak r  ) = 0.
$$
Identify $\frak p \in \frak P $ with $  \gamma ( \frak p ,  \frak p  )$.
Finally  for $\frak q,\frak r \in \frak P$, and for $u\in Y_X^+(\frak q), v \in Y_X^-(\frak r),$  set
$$
[O(u)]_{\approx (D)}[O(v)]_{\approx (D)} = [u]_{\approx\negthinspace(D)}
\gamma ( \frak q ,  \frak r  )[v]_{\approx (D)}.
$$

An isomorphism $\eta_{\sigma, D}$ of $ [Y_X]_{\approx (X)}$ onto 
$ [\Omega(Y^{(D)}_X)]_{\approx (D)}$ is obtained by choosing out of every 
${\approx\negthinspace(X)}$-equivalence class 
$\alpha$ a point  
$\eta (\alpha)\in Y^{(D)}_X$, and by setting 
$$
\eta^{(D)}_X(\alpha) = [\eta (\alpha)]_{\approx (D)}. 
 $$
 
\begin{theorem}
For a subshift $X \subset\Sigma^{\Bbb Z}$ with Property $(A)$ and for $\sigma\in \Sigma,$ 
the semigroups, that are associated to $X$ and $X^{(\sigma)}$,  are isomorphic.
 \end{theorem}
\begin{proof}
Set 
$$
 d^{(\frak p^{(\sigma)})}= \varphi^{(\sigma)}( d^{(\frak p}),\qquad \frak p \in \frak P.
$$
One has
$$
\pi(d^{(\frak p^{(\sigma)})})  =  \pi(d^{(\frak p}), \qquad  \frak p \in \frak P ,
$$
and one has by Lemma 2.2, that
$$
H(d^{(\frak p^{(\sigma)})}) = H(d^{(\frak p}),  \qquad  \frak p \in \frak P.
$$
Setting
$$
D^{(\sigma)} = \{ d^{(\frak p^{(\sigma)})}: \frak p \in \frak P \}.
$$
yields a system of representatives of the $\approx\negthinspace(X^{(\sigma)})$-equivalence classes in 
$\frak P(X^{(\sigma)})$. 
By construction
$$
\varphi^{(\sigma)}(y[u,v]) = y[\varphi^{(\sigma)} (u), \varphi^{(\sigma)}(v)],\quad u, v \in Y_X.
$$
Also, by Lemma 3.1, for $\frak q,\frak r  \in \frak P  $,
$$
Y_X^-(\frak q) \cap A(X) \cap Y_X^-(\frak r) \neq \emptyset,
$$
if and only if
$$
Y_X^-(\frak q^{(\sigma)}) \cap A(X^{(\sigma)})
 \cap Y_X^-(\frak r^{(\sigma)})\neq \emptyset.
$$
It follows that an isomorphism 
$\psi_{\sigma, D}$ of $[Y^{(D)}_X]_{\approx(D)}$ onto  $[Y^{(D^{(\sigma)})}_{X^{(\sigma)}}]_{\approx(D^{(\sigma)})} $ is given by setting
$$
\psi_{\sigma, D}([y]) = [\varphi^{(\sigma)}(y)],
\qquad y \in Y^{(D)}_X,
$$
and one obtains an isomorphism $\Xi^{(\sigma)}$ of 
$[Y_X]_{\approx(X)}$ onto  $[Y_{X^{(\sigma)}}]_{\approx(X^{(\sigma)})} $ by setting
$$
\Xi^{(\sigma)} = \eta^{-1}_{\sigma, D}\psi_{\sigma, D} \eta_{\sigma, D}. \qed
$$ 
\renewcommand{\qedsymbol}{}
\end{proof}

For  the invariance of the associated semigroup under flow equivalenc,  under the assumption that 
$A(X)$ is dense in $X$, or in the sofic case, see also \cite[Theorem 9.20]{CS}).

The semigroup $[Y_X^{(D)}]_{\approx(D)}$ is a set of equivalence classes of orbits.
As originally done in \cite{Kr1}, we have introduced the associated semigroup of a subshift with Property $(A)$ in terms of equivalence classes of points, rather than equivalence classes of orbits.
However, since points in $Y_X$, that are in the same orbit, are $\approx(X)$-equivalent, 
one can define the associated semigroup in the first place as a set of equivalence classes of orbits. The same remark applies to the set of idempotents $\frak P$.
When the associated semigroup is introduced as a set of equivalence classes of orbits, then the mapping $\xi_\sigma$ is seen to induce the isomorphism of the associated semigroup of $X$ onto the associated semigroup of $X^{(\sigma)}$.

\section{$\mathcal R$-graph shifts}

Given finite sets $\mathcal E^-$ and $\mathcal E^+$ and a relation 
$\mathcal R \subset \mathcal E^-  \times  \mathcal E^+  $, we set
$$
\mathcal E^-(\mathcal R) = \{e^-  \in  \mathcal E^-: \{e^-  \}\times \mathcal E^+\subset \mathcal R  \},
\quad\mathcal E^+(\mathcal R) = \{e^+  \in  \mathcal E^+: \mathcal E^-\times \{e^+  \}\subset \mathcal R  \}.
$$
and
\begin{align*}
&\Omega^+_\mathcal R(e^-) = \{e^+ \in  \mathcal E^+ : ( e^-,e^+ ) \in \mathcal R\}, \quad  e^- \in \mathcal E^- , \\
&\Omega^-_\mathcal R(e^+) = \{e^- \in  \mathcal E^-: ( e^-,e^+ ) \in \mathcal R \}, \quad  e^+ \in \mathcal E^+.\end{align*}

We recall from \cite{Kr2} the notion of  an $\mathcal R$-graph.
Let there be given 
a finite  directed graph with vertex set $\frak P$ and edge set $\mathcal E$. Assume also given a partition 
$$
\mathcal E = \mathcal E^-  \cup\mathcal E^+.
$$
With $s$ and $t$ denoting the source  and the target vertex of a directed edge
we set
\begin{align*}
& \mathcal E^- (\frak q,\frak r) = \{ e^- \in  \mathcal E^- : s(e^-) = \frak q,\  t( e^-) =  \frak r \},
\\
& \mathcal E^+(\frak q,\frak r) = \{ e^- \in  \mathcal E^+ : s(e^+) = \frak r,\  t( e^+) =  \frak q \}, \qquad  \frak q,\frak r \in \frak P.
\end{align*}
 We assume that $ \mathcal E^- (\frak q,\frak r)  \neq \emptyset$ if and only if $  \mathcal E^+(\frak q,\frak r) \neq \emptyset,  \frak q,\frak r \in \frak P$, and we assume that 
the  directed graph $(\frak P  ,  \mathcal E^-   )$ is strongly connected, or, equivalently, that 
the  directed graph $(\frak P  ,  \mathcal E^+ )$ is strongly connected.
Let there further be given relations 
$$
\mathcal R    (\frak q,\frak r) \subset   \mathcal E^- (\frak q,\frak r)   \times   \mathcal E^+(\frak q,\frak r) , \qquad \frak q,\frak r \in \frak P,
$$
and set
$$
\mathcal R = \bigcup_{ \frak q,\frak r \in \frak P} \mathcal R    (\frak q,\frak r) .
$$
The resulting structure, for which we use the notation $\mathcal G_\mathcal R(\frak P, \mathcal E^-,\mathcal E^+)$, is called an $\mathcal R$-graph. 

We also recall 
the construction of a semigroup (with zero)  $\mathcal S_{\mathcal R}(\frak P,   \mathcal E^-, \mathcal E^+   )$ 
from an $\mathcal R$-graph $\mathcal G_\mathcal R(\frak P, \mathcal E^-,\mathcal E^+)$ as described in \cite {Kr2}.   
The semigroup $\mathcal S_{\mathcal R}(\frak P,   \mathcal E^-, \mathcal E^+   )$ contains idempotents $\bold 1_{\frak p}, \frak p \in \frak P,$ and  has $\mathcal E$ as a generating set. Besides $\bold 1_{\frak p}^2 = \bold 1_{\frak p},\frak p \in \frak P$, the defining relations are:
$$
f^- g^+ =\bold 1_{\frak q}, \quad f^- \in  \mathcal E^- (\frak q,\frak r), g^+ \in  \mathcal E^+(\frak q,\frak r) , (  f^- ,  g^+ ) \in \mathcal R    (\frak q,\frak r), \quad \frak q,\frak r \in\frak P,
$$
and
\begin{align*}
&\bold 1_{\frak q} e^- = e^- \bold 1_{\frak r} = e^-, \quad e^- \in  \mathcal E^- (\frak q,\frak r), \\
&\bold 1_{\frak r} e^+ = e^+ \bold 1_{\frak q} = e^+, \quad e^+ \in  \mathcal E^+ (\frak q,\frak r), \quad  \frak q,\frak r \in\frak P,
\end{align*}
$$
f^- g^+  = \begin{cases}{ \bold 1}_{{\frak q}}, &\text {if $ (  f^- ,  g^+ ) \in{ \mathcal R}({\frak q},{\frak r}) $,}\\
0, & \text{if $ (  f^- ,  g^+ )\notin {\mathcal R}({\frak q},{\frak r}), \quad f^- \in  {\mathcal E}^- ({\frak q},{\frak r}), g^+ \in { \mathcal E}^+({\frak q},{\frak r)},        
\
{\frak q},  {\frak r} \in{\frak P},
 $}
\end{cases}
$$
and
$$
\bold 1_{\frak q}\bold 1_{\frak r}= 0,    \quad  \frak q,\frak r \in \frak P,\frak q \neq\frak r .
$$
The semigroup $\mathcal S_{\mathcal R}(\frak P,   \mathcal E^-, \mathcal E^+ )  $ is called an $ \mathcal R$-graph semigroup.

The $\mathcal R$-graph shift $M\negthinspace D_{\mathcal R}(\frak P,   \mathcal E^-, \mathcal E^+ )  $ of the 
$\mathcal R$-graph $\mathcal G_\mathcal R(\frak P, \mathcal E^-,\mathcal E^+)$ is the subshift
$$
M\negthinspace D_{\mathcal R}(\frak P,   \mathcal E^-, \mathcal E^+ ) 
 \subset( \{  \mathcal E^- \cup \mathcal E^+)^\Bbb Z
$$
with the admissible words $(\sigma_i)_{1 \leq i \leq I  } , I \in \Bbb N,$ of 
$M\negthinspace D_{\mathcal R}(\frak P,   \mathcal E^-, \mathcal E^+ )$  given by the condition 
\begin{align*}
\prod_{1 \leq i \leq I } \sigma_i \neq 0.
\end{align*}

For an $\mathcal R$-graph $\mathcal G_\mathcal R(\frak P, \mathcal E^-, \mathcal E^+)$ 
we denote by 
$\frak P^{(1)}$ the set of vertices in $\frak P$ that have a single predecessor vertex in $\mathcal E^-$, or, equivalently, that have a single successor vertex in $\mathcal E^+$. For $\frak p \in\frak P^{(1)}$ the predecessor vertex of  $\frak p $ in $\mathcal E^-$, which is identical to the successor vertex of  
$\frak p $ in  $\mathcal E^+$,
is denoted by $\kappa(\frak p)$.
We set
$$
\mathcal E^-_\mathcal R = \bigcup_{\frak p \in \frak P^{(1)}}\mathcal E^-(\mathcal R(\kappa (\frak p), \frak p))  ,
\quad
\mathcal E^+_\mathcal R = \bigcup_{\frak p \in \frak P^{(1)}}\mathcal E^+(\mathcal R(\kappa (\frak p), \frak p)) ,
$$
and
$$
{ \frak P}_\mathcal  R ^{(1)}=
\{ \frak p \in 
\frak P^{(1)}: 
 \mathcal R(\kappa (\frak p), \frak p)    =   \mathcal E^-(\kappa (\frak p), \frak p)  \times \mathcal E^+(\kappa (\frak p), \frak p)   \}.
$$

We formulate conditions $(a)$, $(b)$, $(c)$ and $(d)$ on  an $\mathcal R$-graph
$\mathcal G_{\mathcal R}(\frak P,   \mathcal E^-, \mathcal E^+ )  $ as follows:

\noindent 
\begin{align*}
 \Omega^+_{\mathcal R(\frak q,\frak r)}(e^-) \neq \Omega^+(\tilde e^-),\qquad \quad e^-,\tilde e^-\in  \mathcal E^- (\frak q,\frak r),e^-\neq\tilde e^-, \qquad \frak q,\frak r \in \frak P,
  \tag{$a-$} 
\end{align*}
\begin{align*}
\Omega^-_{\mathcal R(\frak q,\frak r)}(e^+) \neq \Omega^-(\tilde e^+), \qquad \quad e^+,\tilde e^+\in  \mathcal E^+ (\frak r, \frak q),e^+\neq\tilde e^+,\qquad
\frak q,\frak r \in \frak P.   \tag {$a+$} 
\end{align*}
\noindent
$(b-)$ \qquad There is no non-empty  cycle  
in 
$\mathcal E^-_\mathcal R $,

\bigskip
\noindent
$(b+)$ \qquad There is no non-empty cycle 
in $\mathcal E^+_\mathcal R $,

\bigskip
\noindent
$(c)$\qquad  \  \ \  \ For $\frak p \in  \frak P^{(1)}$  such that $\kappa (\frak p ) \neq \frak p$, $\mathcal E^-_\mathcal R(\frak p)= \emptyset,$ or  $\mathcal E^-_\mathcal R(\frak p)= \emptyset$,

\bigskip
\noindent
$(d)$\qquad  \  \ \  \ For $\frak q, \frak r \in  \frak P^{(1)}$ ,  $\frak q  \neq \frak r$, there do not simultaneously exist a path  in $\mathcal E^-_\mathcal R $
 from 
  $\frak q$ to  $\frak r$ 
 and a path in $\mathcal E^+_\mathcal R $ from $\frak q$  to $\frak r$,

\begin{theorem}
For $\mathcal R$-graphs 
$\mathcal G_{\mathcal R}(\frak P,   \mathcal E^-, \mathcal E^+ )  $ that satisfy the Conditions 
$(a)$, $(b)$, $(c)$ and $(d)$
the flow equivalence of the $\mathcal R$-graph shifts 
$D_{\mathcal R}(\frak P,   \mathcal E^-, \mathcal E^+ )  $
 implies the isomorphism of the $\mathcal R$-graphs 
 $\mathcal G_{\mathcal R}(\frak P,   \mathcal E^-, \mathcal E^+ )  $.
\end{theorem}
\begin{proof}
By Theorem 2.3 of \cite {HK} and  Theorem 6.1 of \cite {HK} the conditions imply that the  $\mathcal R$-graph shift 
$D_{\mathcal R}(\frak P,   \mathcal E^-, \mathcal E^+ )  $
has property $(A)$, and that the semigroup, that is associated to it, is 
 $\mathcal S_{\mathcal R}(\frak P,   \mathcal E^-, \mathcal E^+ )$. By Theorem 4.2 the flow equivalence of the shifts $D_{\mathcal R}(\frak P,   \mathcal E^-, \mathcal E^+ )  $ implies the isomorphism of the $\mathcal R$-graph semigroups  
 $\mathcal S_{\mathcal R}(\frak P,   \mathcal E^-, \mathcal E^+ )  $,
 which in turn, by Theorem 2.1  of \cite {Kr2}, implies the isomorphism of the $\mathcal R$-graphs 
 $\mathcal G_{\mathcal R}(\frak P,   \mathcal E^-, \mathcal E^+ )$.
\end{proof}

Theorem 5.1 extends the result of Costa and Steinberg, that the flow equivalence of Markov-Dyck shifts of finite irreducible  directed graphs, in which every vertex has at least two incoming edges, implies the isomorphism of the graphs
(see \cite[Theorem 8.6]{CS}),

Let for $K > 1$, $B_K$ denote the full shift on $K$ symbols, and let $D_2$ denote the Dyck shift on four symbols. The shifts $D_2\times B_K, K > 1,$ belong to the class of $\mathcal R$-graph shifts. 
They arise from the one-vertex $\mathcal R$-graphs  
$\mathcal G_{\mathcal R}(\{\frak p\},   \mathcal E^-, \mathcal E^+ )$, where
\begin{align*}
\mathcal E^- = \{  e^-( m,\beta):1\leq m \leq K ,\beta = 0, 1   \}, \
\mathcal E^+ =\{  e^-( l,\beta):1\leq l \leq K ,\beta = 0, 1   \},
\end{align*}
and where
$$ (e^-m,(\beta^-) ,    e^-(l,\beta^+)) \in \mathcal R( \frak p,\frak p ),$$
if and only if $$ \beta^-  =  \beta^+ , \quad 1\leq m \leq K, 1\leq l \leq K. $$
These $\mathcal R$-graphs do not satisfy the conditions of Theorem 5.1, but the $ \mathcal R$-graph shifts $D_2\times B_K, K > 1,$
have Property $(A)$,
and the flow equivalence of their  $\mathcal R$-graph shifts $D_2\times B_K, K > 1,$ still implies  the isomorphism of these $\mathcal R$-graphs.
This can be seen from the invariance under flow equivalence of the K-groups of subshifts as shown by Matsumoto in \cite {M1}, and from
$$
K_0(D_2\times B_K ) = \Bbb Z [\tfrac{1}{n}]^\infty, \quad  K > 1,
$$
as also shown by Matsumoto \cite[Section 8]{M2}.
Note that the associated semigroup of $D_2\times B_K, K > 1,$  is the Dyck inverse monoid $\mathcal D_2$.

\par\noindent Wolfgang Krieger
\par\noindent Institute for Applied Mathematics, 
\par\noindent  University of Heidelberg,
\par\noindent Im Neuenheimer Feld 294, 
 \par\noindent 69120 Heidelberg,
 \par\noindent Germany
\par\noindent krieger@math.uni-heidelberg.de


\begin{thebibliography}{9999}

 \bibitem[CS]{CS}
{\sc   A.~Costa and B.~Steinberg},
{\it  A categorical invariant of flow equivalence of shifts},
 arXiv: 1304.3487 [math.DS]
 
 \bibitem[HK] {HK}
{\sc T.~Hamachi and W.~Krieger},
{\it A construction of subshifts and a class of semigroups},
 arXiv: 1303.4158 [math.DS]      

\bibitem[Ki]{Ki}{\sc B.~P.~Kitchens},
{\it Symbolic dynamics}, Springer, Berlin, Heidelberg, New York
(1998)

\bibitem[Kr1]{Kr1}
{\sc  W.~Krieger},
{\it  On a syntactically defined invariant of symbolic dynamics},
Ergod. Th. \& Dynam. Sys.
  {\bf  20}
(2000),
501 -- 516

\bibitem[Kr2]{Kr2}
{\sc  W.~Krieger},
{\it  On subshift presentations},
 arXiv: 1209.2578 [math.DS]
 
 \bibitem[LM]{LM}{\sc D.~Lind and B.~Marcus},
{\it An introduction to symbolic dynamics and coding},
 Cambridge University Press, Cambridge
(1995)

 \bibitem [M1]{M1}
{\sc K.~Matsumoto},
{\it  Bowen-Franks groups as an invariant for flow equivalence of subshifts},
Ergod. Th. \& Dynam. Sys. 
{\bf  21}
(2001),
1831 -- 1842
 
 \bibitem [M2]{M2}
{\sc K.~Matsumoto},
{\it K-theoretic invariants and conformal measures of the Dyck shifts},
International J. of Mathematics
{\bf  16}
(2005),
 213 -- 248
 
  \bibitem [M3]{M3}
{\sc K.~Matsumoto},
{\it  C*-algebras arising from Dyck systems of topological Markov chains},
Math. Scand. 
{\bf  109}
(2011),
 31 -- 54

  \bibitem [PS]{PS}
{\sc B.~Parry and D.~Sullivan},
{\it  A topological invariant for flows on one-dimensional spaces},
Topology 
{\bf  14}
(1975),
 297 -- 299

\end{thebibliography}
 \end{document}